\pdfoutput=1
\documentclass[12pt,reqno,final]{amsart}
\usepackage[color,notref]{showkeys}
\usepackage{layout,newcommands,amscd}
\usepackage{hyperref}
\hypersetup{pdfpagemode=UseOutlines,colorlinks=false,pdfpagelayout=SinglePage,pdfstartview=FitH,bookmarksopen=true}

\title{Matched pairs of Courant algebroids}
\author{Melchior Gr{\"u}tzmann}
\address{Department of Mathematics, Sun Yat-sen University \newline 
Guangzhou 510275, People's Republic of China}
\curraddr{Department of Mathematics, Northwestern Polytechnical University \newline 
Xi'an 720071, People's Republic of China}
\email{\href{mailto:melchiorg@gmail.com}{melchiorg@gmail.com}}
\author{Mathieu Sti{\'e}non}
\address{Department of Mathematics, Penn State University \newline
University Park, PA 16802, United States of America}
\email{\href{mailto:stienon@math.psu.edu}{stienon@math.psu.edu}}
\thanks{Research partially supported by NSA grant H98230-12-1-0234 and NSF grant DMS0605725 (US-China collaboration supplement).}
\newcommand\defbb[2]{\def#1{{\mathbb{#2}}}}
\newcommand\defbf[2]{\def#1{{\boldsymbol{#2}}}}
\newcommand\defcal[2]{\def#1{{\mathcal{#2}}}}
\newcommand\deffrak[2]{\def#1{{\mathfrak{#2}}}}
\newcommand\defrm[2]{\def#1{{\mathrm{#2}}}}

\def\<{\langle}
\def\>{\rangle}
\def\:{\colon}
\def\a{\alpha}
\def\b{\beta}
\def\c{\gamma}
\def\conn_#1{\nabla_{\! #1}\,}  % there is a spacing problem with the nabla symbol
\defbb\C{C}
\def\smooth{{C^\infty}}
\defrm\ud{d}
\defcal\D{D}
\def\dia{\diamond}
\def\ConnectionDer_#1{\Dot\nabla_{\!\!#1}\,}

\defcal\E{E}
\deffrak\g{g}
\DeclareMathOperator\Graph{Graph}
\def\ins{i}  %% insertion of a vector into a form
\defcal\L{L}
\def\lconn_#1{\nablaleft_{\!\!#1}\,}

\deffrak\o{o}
\def\oo{\Omega}

\defbb\R{R}
\def\rconn_#1{\nablaright_{\!\!#1}\,}

\def\w{\mho}
\defbf\X{X}
\deffrak{\XX}{X} % vector fields

\begin{document}

\begin{abstract}  We introduce the notion of matched pairs
 of Courant algebroids and give several examples
 arising naturally
 from complex manifolds, holomorphic Courant algebroids, and
certain regular Courant algebroids. 
We consider the matched sum of two Dirac subbundles, one in each of
 two Courant algebroids forming a matched pair.
\end{abstract}

\maketitle

\section{Introduction}

Matched pairs of algebraic structures occur naturally in several contexts of mathematics.
For instance, a matched pair of groupoids, introduced by Mackenzie in~\cite{Mack92} 
while studying double (Lie) groupoids, are two groupoids $G\rightrightarrows M$ 
and $H\rightrightarrows M$ over the same base $M$ together with a representation of $G$ on $H$ 
and a representation of $H$ on $G$ compatible such that their product $G\bowtie H$ is again a groupoid.
The infinitesimal version is a matched pair of Lie algebroids,
which were introduced by Lu in~\cite{Lu97} and studied by Mokri in~\cite{Mor97}. 
They consist of two Lie algebroids $A_1$ and $A_2$ over the same base manifold $M$,
together with an $A_1$-module structure on $A_2$ and an $A_2$-module structure on $A_1$, 
such that their direct sum $A_1\bowtie A_2$ is again a Lie algebroid.
Further examples arise from the study of holomorphic Poisson structures 
and holomorphic Lie algebroids~\cite{LSX08}.

The main goal of this note is to study matched pairs of Courant algebroids.
More precisely, we investigate the question under which conditions the direct sum
of two Courant algebroids over the same base manifold is still a Courant algebroid.
We derive conditions which are similar to those of Lu and Mokri~\cite{Lu97,Mor97}.
However, there is a significant difference between matched pairs of
Courant algebroids and matched pairs of Lie algebroids. 
It turns out, unlike Lie algebroids, that each component of the direct sum Courant algebroid 
of a matched pair is no longer a Courant subalgebroid.

Examples of matched pairs of Courant algebroids have appeared in literature. 
In connection with the study of port-Hamiltonian systems, 
Merker considered the Courant algebroid $TM\oplus T^*M\oplus E\oplus E^*$, 
where $E\to M$ be a vector bundle endowed with a flat connection $\nabla$~\cite{Mer09}.
This is indeed a very simple example of matched pairs of Courant algebroids.

Another class of matched pairs of Courant algebroids arise 
when studying holomorphic Cou\-rant algebroids along a similar line 
as in the study of holomorphic Lie algebroids~\cite{LSX08}.
In particular, we prove that a holomorphic Courant algebroid 
over a complex manifold $X$ is equivalent to a matched pair 
of (smooth) Courant algebroids satisfying certain special properties,
one of which is the standard Courant algebroid $T_X^{0, 1}\oplus (T_X^{0, 1})^*$.

A third class of examples arise from the construction of regular Courant algebroids, 
which were recently classified by one of the authors in a joint work~\cite{CSX09}.

The paper is organized as follows. In Section~\ref{s:prelim} we review the notion of Courant algebroids 
and matched pairs of Lie algebroids. In Section~\ref{s:mp} we introduce the definition of matched pairs 
of Courant algebroids. In Section~\ref{s:ex} we give 
four classes of examples: Courant algebroids with flat connections, complex manifolds, holomorphic Courant algebroids, 
and flat regular Courant algebroids. In Section~\ref{s:mpDirac}, we give a definition 
of matched pairs of Dirac structures and show that a matched pair of Dirac structures is 
a matched pair of Lie algebroids. A supergeometric description of matched pairs of Courant algebroids 
will be discussed elsewhere.

We would like to thank Thomas Strobl and Ping Xu for enlightening discussions. 

\section{Preliminaries}\label{s:prelim}
We recall the definition of Courant algebroids based on the Dorfman bracket originally introduced in~\cite{Dor87,Dor93}. For a comparison to the bracket introduced by Courant~\cite{Xu97}, see~\cite{Royt99,Cour90}.

\begin{defn} 
A (real) \emph{Courant algebroid} is a real vector bundle $E\to M$ endowed with a symmetric non-degenerate $\RR$-bilinear form $\langle~,~\rangle$ on $E$ with values in $\RR$, an $\R$-bilinear product $\dia$ on the space of sections $\sections{E}$ called Dorfman bracket and a bundle map $\rho:E\to TM$ (over the identity) called anchor map satisfying 
\begin{gather}
\phi\dia(\phi_1\dia\phi_2) = \phi_1\dia(\phi\dia\phi_2) +(\phi\dia\phi_1)\dia\phi_2\label{Jacobi}  \\
\phi\dia(f\phi') = \big(\rho(\phi)f\big)\phi' +f(\phi\dia\phi')  \label{Leibniz}  \\
\phi\dia\phi = \thalf\D\<\phi,\phi\>  \label{nSkew}  \\
\rho(\phi)\<\phi',\phi'\> = 2\<\phi\dia\phi',\phi'\> , \label{adInvar}
\end{gather} 
where $\phi,\phi_1,\phi_2,\phi'\in\sections{E}$, $f\in\smooth(M)$ and 
$\D:\smooth(M)\to\sections{E}$ is defined by the relation 
$\<\D f,\phi\>=\rho^*(\ud f)\phi$.
\end{defn}

For any $\phi,\psi\in\sections{E}$, we have~\cite{Uchi02}
\beq{rhomor}\rho(\phi\dia\psi)=[\rho(\phi),\rho(\psi)] .\eeq 
Moreover $(\D f)\dia\phi=0$.

\begin{rmk}
Complex Courant algebroids are defined similarly except that the pairing is $\CC$-valued, 
the anchor is $TM\otimes\CC$-valued, and $\CC$-linearity replaces $\RR$-linearity. 
\end{rmk}

In one of his letters to Alan Weinstein, Pavol \v{S}evera described the following example: 
%\cite{SevLett}

\begin{ex}\label{twisted}
To each closed 3-form $H$ on $M$ is associated a Courant algebroid structure on $TM\oplus T^*M$ with 
inner product 
\beq{inner} \<X\oplus\alpha,Y\oplus\beta\> = \alpha(X)+\beta(Y) ,\eeq
anchor map
\beq{anchor} \rho(X\oplus\alpha) = X ,\eeq 
and Dorfman bracket
\beq{bracket} (X\oplus\a)\dia(Y\oplus\b) = [X,Y]\oplus\big( \L_X\b-\ins_Y\ud\a + \ins_{X\wedge Y} H\big) ,\eeq
where $X,Y\in\XX(M)$ and $\a,\b\in\OO^1 (M)$.
This Dorfman bracket is called standard if $H=0$ and $H$-twisted if $H\ne 0$. 
\end{ex}

\begin{defn}  Given an anchored vector bundle $E\xrightarrow{\rho}TM$ 
and a vector bundle $V$ over a smooth manifold $M$, an $E$-connection on $V$ 
is a bilinear operator $\nabla:\sections{E}\otimes\sections{V}\to\sections{V}$ fulfilling
\begin{gather}
  \conn_{f\psi}v = f\conn_\psi v , \label{dog} \\
  \conn_\psi(fv) = \big(\rho(\psi)f\big) v + f \conn_\psi v \label{cat} 
\end{gather}
for all $f\in\smooth(M)$, $\psi\in\sections{E}$, and $v\in\sections{V}$.
\end{defn}

\begin{defn}[\cite{Mor97}] Two Lie algebroids $A$ and $A'$ form a matched pair 
when a flat $A$-connection $\nablaright$ on $A'$ and a flat $A'$-connection $\nablaleft$ on $A$ 
satisfying 
\begin{gather}
  \lconn_\a[b,c] = [\lconn_\a b,c] +[b, \lconn_\a c] 
    +\lconn_{\rconn_c\a}b -\lconn_{\rconn_b\a}c \label{crocodile} \\
  \rconn_a[\b,\c] = [\rconn_a \b,\c] +[\b, \rconn_a \c] 
    +\rconn_{\lconn_\c a}\b -\rconn_{\lconn_\b a}\c \label{alligator}
\end{gather}
(for all $\a,\b,\c\in\sections{A'}$ and $a,b,c\in\sections{A}$) are specified. 
\end{defn}

The following theorem is due to  Mokri.

\begin{thm}[\cite{Mor97}]
Let $A$ and $A'$ be a pair of Lie algebroids (with anchors $\rho_A$ and $\rho_{A'}$, and Lie brackets $[,]_A$ and $[,]_{A'}$, resp.). 
If a pair of connections $\nablaleft$ and $\nablaright$ makes $(A,A')$ into
 a matched pair of Lie algebroids, then the vector bundle $A\oplus A'$ is a Lie algebroid when endowed with the anchor map $\rho_A + \rho_{A'}$ and the bracket 
\[ [a+\alpha,b+\beta]= ( [a,b]_A + \lconn_{\alpha} b - \lconn_{\beta} a )
+ ( [\alpha,\beta]_{A'} + \rconn_a \beta - \rconn_b \alpha ) .\] 
Conversely, given a Lie algebroid $L$ and two Lie subalgebroid $A$ and $A'$ such that $L=A\oplus A'$ as vector bundles, then $(A,A')$ is a matched pair of Lie algebroids whose pair of connections is determined by the following relation: 
\[ [a,\beta]=-\lconn_{\beta}a+\rconn_a\beta .\]
\end{thm}

\section{Matched pairs}\label{s:mp}
Let $(E,\ip{}{},\rho,\db{}{})$ be a Courant algebroid. 
Assume we are given two subbundles $E_1,E_2$ of $E$ such that $E=E_1\oplus E_2$ 
and $E_1^\perp=E_2$. 
Let $\pr_1$ (resp.\ $\pr_2$) denote the projection of $E$ onto $E_1$ (resp.\ $E_2$) and 
$\inj_1$ (resp.\ $\inj_2$) denote the inclusion of $E_1$ (resp.\ $E_2$) 
into $E$, respectively. 
Assume that $E_k$ is itself a Courant algebroid with anchor $\rho_k=\rho\circ\inj_k$, 
inner product $\ip{a}{b}_k=\ip{\inj_k a}{\inj_k b}$, and Dorfman bracket 
$\dbi{a}{b}{k}=\pr_k(\db{\inj_k a}{\inj_k b})$.
A natural question is how to recover the Courant algebroid
structure on $E$ from  $E_1$ and $E_2$.

\begin{prop}  The inner product and the anchor map of $E$ are uniquely 
determined by their restrictions to the subbundles $E_1$ and $E_2$.  Indeed, for $a,b\in\sections{E_1}$ 
and $\alpha,\beta\in\sections{E_2}$, we have 
\begin{gather}  
\<a\oplus\alpha,b\oplus\beta\> = \<a,b\>_1 +\<\alpha,\beta\>_2 , \label{bull} \\
\rho(a\oplus\alpha) = \rho_1(a)+\rho_2(\alpha) , \label{bear}
\end{gather}
where $a\oplus\alpha$ is shorthand for $\inj_1(a)+\inj_2(\alpha)$.
\end{prop}

\begin{prop} 
The bracket on $E$ induces an $E_1$-connection on $E_2$: 
\beq{bird} \rconn_{a} \beta=\pr_2\big(\db{(\inj_1 a)}{(\inj_2 \beta)}\big) \eeq
and an $E_2$-connection on $E_1$: 
\beq{fish} \lconn_{\beta} a=\pr_1\big(\db{(\inj_2 \beta)}{(\inj_1 a)}\big) .\eeq
These connections preserve the inner products on $E_1$ and $E_2$: 
\begin{gather}
\rho(\alpha)\ip{a}{b}_1= \ip{\lconn_{\alpha}a}{b}_1 +\ip{a}{\lconn_{\alpha}b}_1 \label{cow} \\
\rho(a)\ip{\a}{\b}_2= \ip{\rconn_a \a}{\b}_2 +\ip{\a}{\rconn_{a}\b}_2 \label{calf}
\end{gather}

Moreover, we have 
\beq{parrot} (a\oplus0)\dia(0\oplus\beta) = -\lconn_{\beta} a \oplus \rconn_{a} \beta \eeq
for $a\in\sections{E_1}$ and $\beta\in\sections{E_2}$.
\end{prop}

\begin{proof} The first two equations follow from Leibniz rule \eqref{Leibniz}. 
The next two equations follow from ad-invariance \eqref{adInvar}. 
The last equation uses in addition axiom~\eqref{nSkew}.
\end{proof}

\begin{prop} 
\begin{enumerate}
\item For all $a,b\in\sections{E_1}$, we have 
\beq{pig} \db{(a\oplus0)}{(b\oplus0)} = (\dbi{a}{b}{1}) \oplus 
\big( \tfrac{1}{2} \DD_2 \ip{a}{b}_1 + \oo(a,b) \big) ,\eeq 
where $\oo\:\wedge^2\sections{E_1}\to \Gamma(E_2)$ is defined by the relation 
\beq{monkey} \oo(a,b)=\thalf \pr_2 (\db{\inj_1 a}{\inj_1 b}-\db{\inj_1 b}{\inj_1 a}) .\eeq 
In fact $\oo$ is entirely determined by the connection 
$\nablaleft\:\sections{E_2}\otimes\sections{E_1}\to\sections{E_1}$ through the relation 
\beq{donkey} \ip{\gamma}{\oo(a,b)}_2 = \thalf (\ip{\lconn_{\gamma} a}{b}_1 
- \ip{a}{\lconn_{\gamma} b}_1) \eeq 
\item For all $\alpha,\beta\in\sections{E_2}$, we have 
\beq{snake} \db{(0\oplus\alpha)}{(0\oplus\beta)} = 
\big( \tfrac{1}{2} \DD_1 \ip{\alpha}{\beta}_2 + \w(\alpha,\beta)\big) 
\oplus (\dbi{\alpha}{\beta}{2}) ,\eeq
where $\w\:\wedge^2\sections{E_2}\to \sections{E_1}$ is defined by the relation 
\beq{elephant} \w(\alpha,\beta)=\thalf \pr_1 (\db{\inj_2 \alpha}{\inj_2 \beta} 
- \db{\inj_2 \beta}{\inj_2 \alpha}) .\eeq 
In fact, $\w$ is entirely determined by the connection 
$\nablaright:\sections{E_1}\otimes\sections{E_2}\to\sections{E_2}$
 through the relation 
\beq{shark} \ip{c}{\w(\alpha,\beta)}_1 = \thalf (\ip{\rconn_{c} \alpha}{\beta}_2 
- \ip{\alpha}{\rconn_{c} \beta}_2) .\eeq 
\end{enumerate}
\end{prop}
\begin{proof}  From axiom \eqref{nSkew} we conclude that the symmetric part of the bracket is given by
\[ (a\oplus0)\dia(a\oplus0)=\tfrac12\D\<a,a\>_1=\tfrac12\D_1\<a,a\>_1\oplus\tfrac12\D_2 \<a,a\>_1 .\] 
Moreover using \eqref{adInvar} we get
\begin{equation*}
  \<0\oplus\gamma,(a\oplus0)\dia(b\oplus0)\> = \rho(a\oplus0)\<0\oplus\gamma,b\oplus0\> 
-\<(a\oplus0)\dia(0\oplus\gamma),(b\oplus0)\> 
   =\<\lconn_\gamma a,b\>_1
\end{equation*}
The formula for $\w$ occurs under analog considerations for $0\oplus\alpha$ and $0\oplus\beta$.
\end{proof}

As a consequence,  we obtain the formula 
\begin{multline} 
\db{(a\oplus\alpha)}{(b\oplus\beta)} 
= \big( \dbi{a}{b}{1} +\lconn_{\alpha} b-\lconn_{\beta} a 
+\w(\alpha,\beta)+\thalf\DD_1\ip{\alpha}{\beta}_2 \big) \\
\oplus \big( \dbi{\alpha}{\beta}{2} +\rconn_{a} \beta-\rconn_{b} \alpha 
+\oo(a,b)+\thalf\DD_2\ip{a}{b}_1 \big)  \label{fullBracket}
,\end{multline}
which shows that the Dorfman bracket on $\sections{E}$ can be recovered from 
the Courant algebroid structures on $E_1$ and $E_2$ together with the 
connections  $\nablaleft$ and $\nablaright$.

\begin{lem}\label{l:quark}
For any $f\in\smooth(M)$, $b\in\sections{E_1}$, and $\b\in\sections{E_2}$,
 we have $\rconn_{\DD_1 f} \beta =0$ and $\lconn_{\DD_2 f} b =0$.
\end{lem}

\begin{proof} 
We have $ \rconn_{\D_1 f}\beta = \pr_2((\D_1 f\oplus0)\dia(0\oplus\beta)) = \pr_2(\D f\dia(0\oplus\beta)) = 0 $.
\end{proof}
Set
\begin{align}\label{dolphin} \Rright(a,b)\alpha &:= \rconn_{a}\rconn_{b} \alpha 
-\rconn_{b}\rconn_{a} \alpha -\rconn_{\dbi{a}{b}{1}}\alpha ,\\
  \label{wale} \Rleft(\alpha,\beta)a &:= \lconn_\alpha\lconn_\beta a 
-\lconn_\beta \lconn_\alpha a -\lconn_{\dbi{\alpha}{\beta}{2}}a  \;.
\end{align} 
\begin{lem}  The curvatures $\Rright$ and $\Rleft$ are sections of $\wedge^2 E_1^*\otimes\mathfrak{o}(E_2)
\cong\wedge^2 E_1^*\otimes\wedge^2 E_2^*$, where $\mathfrak{o}(E_2)$ 
is the bundle of skew-symmetric endomorphisms of $(E_2,\<~,~\>_2)$.
\end{lem}
\begin{proof}  The $\smooth(M)$-linearity of $\Rright(a,b)\alpha$ in $\alpha$ follows from \eqref{rhomor}. 
The curvature $\Rright$ is also $\smooth(M)$-linear in $b$. 
In view of Lemma \ref{l:quark},  it is also skew-symmetric
with respect to   $a$ and $b$.
\end{proof}

\begin{thm}\label{glass} Assume we are given two Courant algebroids $E_1$ and $E_2$ 
over the same manifold $M$ and two connections 
$\nablaright\:\sections{E_1}\otimes\sections{E_2}\to\sections{E_2}$
 and $\nablaleft\:\sections{E_2}\otimes\sections{E_1}\to\sections{E_1}$ preserving the fiberwise
 metrics and satisfying $\rconn_{\DD_1 f} \beta =0$ and $\lconn_{\DD_2 f} b =0$ for all $f\in\smooth(M)$, 
$b\in\sections{E_1}$, and $\b\in\sections{E_2}$.
Then the inner product \eqref{bull}, the anchor map \eqref{bear}, and the bracket \eqref{fullBracket} on the direct sum 
$E=E_1\oplus E_2$ satisfy \eqref{Leibniz}, \eqref{nSkew}, and \eqref{adInvar}. 
Moreover, the Jacobi identity \eqref{Jacobi} is fully equivalent to the following group of properties:
\begin{gather}
  \begin{split}&\lconn_\a(a_1\dia_1 a_2)-(\lconn_\a a_1)\dia_1 a_2-a_1\dia_1(\lconn_\a a_2)%- \\&
    -\lconn_{\rconn_{a_2}\a}a_1 +\lconn_{\rconn_{a_1}\a}a_2 \\
    &\qquad =-\w (\a,\oo(a_1,a_2)+\thalf\D_2\<a_1,a_2\>)
      -\thalf \D_1\big\<\a,\oo(a_1,a_2)+\thalf \D_2\<a_1,a_2\>\,\big\>
  \end{split}\label{derofBr1}\\
  \begin{split}&\rconn_{a}(\a_1\dia_2\a_2)-(\rconn_{a} \a_1)\dia_2\a_2-\a_1\dia_2(\rconn_{a} \a_2)
    -\rconn_{\lconn_{\a_2}a}\a_1 +\rconn_{\lconn_{\a_1}a}\a_2 \\
    &\qquad=-\oo(a,\w (\a_1,\a_2)+\thalf \D_1\<\a_1,\a_2\>)
      -\thalf \D_2\big\<a,\w (\a_1,\a_2)+\thalf \D_1\<\a_1,\a_2\>\,\big\>
  \end{split}\label{derofBr2}\\
  \Rright+\Rleft = 0\label{curvCompat}\\
  \lconn_{\oo(a_1,a_2)} a_3 +\text{c.p.}=0 \label{new4X} \\
  \rconn_{\w (\a_1,\a_2)} \a_3 +\text{c.p.}=0  \label{new4Y}
\end{gather}
\end{thm}
\begin{proof}  Axioms \eqref{Leibniz}, \eqref{nSkew}, and \eqref{adInvar} are straightforward computations 
from the definition \eqref{fullBracket} using \eqref{donkey} and \eqref{shark}.
To check the Jacobi identity, it is helpful to compute for triples of sections of $E_1$ and $E_2$, respectively, 
which gives 8 cases.  Equations
 \eqref{curvCompat}, \eqref{new4X}, and \eqref{new4Y} arise when  pairing the Jacobiator
  with an arbitrary section of $E_1\oplus E_2$. 
\end{proof}

We are now ready to introduce the main notion of the paper.

\begin{defn}\label{pair}
Two Courant algebroids $E_1$ and $E_2$ over the same manifold $M$ 
together with a pair of connections satisfying the properties listed in Theorem~\ref{glass} 
are said to form a \emph{matched pair}. The induced Courant algebroid structure 
on the direct sum vector bundle $E_1\oplus E_2$ is called the \emph{matched sum} 
of the pair $(E_1,E_2)$.
\end{defn}

\section{Examples}\label{s:ex}

\subsection{Courant algebroids with a flat connection}

The first example goes back to Merker in~\cite{Mer09}.  Start with a
 Courant algebroid $(E_1,\dia_1,\rho_1)$ and assume that $\nabla$ is a metric connection 
on a pseudo Euclidean vector bundle $(V,\<~,~\>)$ over the same manifold.  If, in addition, $\conn_{\D f}s=0$
 for all smooth functions $f\in\smooth(M)$, then the curvature $R$, defined as
 $$ R(\psi_1,\psi_2)v = \conn_{\psi_1}\conn_{\psi_2}v -\conn_{\psi_2}\conn_{\psi_1}v - \conn_{\psi_1\dia\psi_2}v
 $$  is an element of $\Gamma(\wedge^2E_1\otimes\mathfrak{o}(V))$.  We require this to vanish.
We can endow $E_2=V$ with the trivial Courant bracket and trivial anchor.
Furthermore,  we assume that the $E_2$-connection on $E_1$ is
trivial.  Then,
\begin{equation}\label{PCour} 
(\psi\oplus v)\dia(\psi'\oplus v') = \big((\psi\dia_1\psi') +\tfrac12\D\<v,v'\>+\oo(v,v')\big) 
\oplus \big(\conn_\psi v'-\conn_{\psi'} v\big) 
.\end{equation}

This Courant algebroid plays an important role in the study of
port-Hamiltonian systems  \cite{Mer09}.
Let $E\to M$ be a  vector bundle  endowed 
with a  flat connection $\nabla$.  The port-Hamiltonian
system can be   described by a Dirac structure $D\subset CM\oplus V$,
 where $CM:=TM\oplus T^*M$ is the standard Courant algebroid, $V:=E\oplus E^*$ 
the trivial Courant algebroid with vanishing bracket and anchor, and $CM\oplus V$ their matched pair as explained above.  
An interesting family of Dirac structures that does not-necessarily project to Dirac structures on $CM$ 
or $V$ arises from a 2-form $\omega\in\Omega^2(M)$ and a so-called port map, which is a bundle map 
$A\:T^*M\to E$.  The Dirac structure is now the graph of the bundle map 
\[ TM\oplus E^*
\xto{\begin{pmatrix}
\omega^\#& -(A\circ\omega^\#)^* \\ A\circ\omega^\#&0 
\end{pmatrix}} 
T^*M\oplus E .\] 
The integrability conditions are $\ud\omega=0$ and $\ud_\oplus(A\circ\omega^\#)=0$, 
where $\ud_\oplus$ is the Lie algebroid differential of the sum Lie algebroid
$TM\oplus E$  of the matched pair of Lie algebroids $(TM, E)$.

Conversely,  given a bivector field $\pi\in\sections{\wedge^2TM}$ 
and a port map $A: T^*M\to E$, we can consider the graph of
the bundle map $\begin{pmatrix}\pi & -A^*\\ A & 0\end{pmatrix}: T^*M\oplus E\to 
T^*M\oplus E$.  This is always isotropic.
It   is integrable iff $[\pi,\pi]=0$, $[\pi,A]_\oplus=0$, and
 $[A,A]_\oplus=0$, where $[~,~]_\oplus$ is the Schouten bracket of the sum
Lie algebroid of the  matched pair $(TM,  E)$.

\subsection{Complex manifolds}\label{ss:cmfd}
Let $X$ be a complex manifold. Its tangent bundle $T_X$ is a holomorphic vector bundle. 
Consider the almost complex structure $j\:T_X\to T_X$.  Since $j^2=-\id$, the complexified tangent bundle $T_X\otimes\C$ 
decomposes as the direct sum of $T^{0,1}_X$ and $T^{1,0}_X$.  Let $\pr\zo\:T_X\otimes\C\to T\zo_X$ and 
$\pr\oz\:T_X\otimes\C\to T\oz_X$ denote the canonical projections.  The complex vector bundles $T_X$ and $T_X\oz$ 
are canonically identified with one another by the bundle map $\thalf(\id-\sqrt{-1} j):T_X\to T_X\oz$. 

\begin{defn} A \emph{complex Courant algebroid} is a complex vector bundle $E\to M$ endowed 
with a symmetric nondegenerate $\CC$-bilinear form $\ip{.}{.}$ on the fibers of $E$ with values in $\CC$, 
a $\CC$-bilinear Dorfman bracket $\db{}{}$ on the space of sections $\sections{E}$ 
and an anchor map $\rho:E\to TM\otimes\CC$ satisfying relations 
\eqref{Jacobi}, \eqref{Leibniz}, \eqref{nSkew}, and \eqref{adInvar}, 
where $\D:\smooth(M;\CC)\to\sections{E}$ is defined by $\<\D f,\phi\>=\rho^*(\ud f) \phi$, 
$\forall\phi\in\sections{E}, f\in\smooth(M;\CC)$
\end{defn}

The complexified tangent bundle $T_X\otimes\CC$ of a complex manifold $X$ is a smooth complex Lie algebroid. 
As a vector bundle, it is the direct sum of $T_X\zo$ and $T_X\oz$, which are both Lie subalgebroids. 

The Lie bracket of (complex) vector fields induces a $T_X\oz$-module structure on $T_X\zo$: 
\beq{horse} \rep_X Y=\przo \lie{X}{Y} \qquad (X\in\XX\oz, Y\in\XX\zo) ,\eeq
and a $T_X\zo$-module structure on $T_X\oz$: 
\beq{deer} \rep_Y X=\proz \lie{Y}{X} \qquad (X\in\XX\oz, Y\in\XX\zo) .\eeq
The flatness of these connections is a byproduct of the integrability of $j$.
We use the same symbol $\rep$ to denote the induced connections on the dual spaces: 
\begin{gather} \rep_X \beta=\przo (\mathcal{L}_X \beta) \qquad (X\in\XX\oz, \beta\in\OO\zo) ,\label{squirrel} \\ 
\rep_Y \alpha=\proz (\mathcal{L}_Y \alpha) \qquad (Y\in\XX\zo, \alpha\in\OO\oz) .\label{frog}
\end{gather}

As in Example~\ref{twisted},  associated to
 each $H^{3,0}\in\Omega^{3,0}(X)$
 such that $\partial H^{3,0}=0$, there is a complex twisted
 Courant algebroid structure on $C\oz_X=T\oz_X\oplus(T\oz_X)^*$.
  Moreover, if two $(3,0)$-forms $H^{3,0}$ and $H'^{3,0}$ are $\partial$-cohomologous, then the associated twisted
 Courant algebroid structures on $C\oz_X$ are isomorphic.
Similarly, associated to each $H^{0, 3}\in\Omega^{0, 3}(X)$ 
 such that $\bar{\partial} H^{0, 3}=0$, there is a complex twisted
 Courant algebroid structure on 
$C\zo_X=T_X\zo\oplus(T_X\zo)^*$.

\begin{prop} 
Let $H=H^{3,0}+H^{2, 1}+H^{1, 2}+H^{0, 3}\in \Omega^3(X)\otimes\CC$,
where $H^{i, j}\in \Omega^{i, j}(X)$, be a closed $3$-form.
Let $(C^{1,0}_X)_{H^{3,0}}$ be the complex
 Courant algebroid structure on $C\oz_X=T\oz_X\oplus(T\oz_X)^*$
twisted by $H^{3,0}$, and $(C^{0, 1}_X)_{H^{0, 3}}$ be the complex 
 Courant algebroid structure on
$C\zo_X=T_X\zo\oplus(T_X\zo)^*$ twisted by $H^{0, 3}$.
Then $(C^{1,0}_X)_{H^{3,0}}$ and $(C^{0, 1}_X)_{H^{0, 3}}$ 
form a matched pair of Courant algebroids, with
connections given by
\begin{align}
\rconn_{X\oplus\alpha}(Y\oplus \beta)
  =&\, \nabla^o_X Y \oplus \nabla^o_X \beta
    +H^{1, 2}(X,Y, \cdot) \label{plane} \\
  \lconn_{Y\oplus\beta}(X\oplus \alpha)
  =&\, \nabla^o_Y X \oplus \nabla^o_Y \alpha
    +H^{2, 1}(Y,X, \cdot) \label{car} 
\end{align} 
for all $X\in\XX^{1,0}$, $\alpha\in\OO^{1,0}$, $Y\in\XX^{0,1}$.
The resulting matched sum Courant algebroid is isomorphic to
the standard complex Courant algebroid $(T_X\oplus T_X^*)\otimes \CC$
twisted by $H$.
\end{prop}

\begin{proof} The Courant algebroid $(C\oz_X)_{H^{3,0}}$ is the twisted
complex Courant algebroid $(T\oz_X \oplus {T\oz_X}^*)_{H^{3,0}}$.
Therefore the 3-form $H^{3,0}$ must be closed under $\partial$. 
This is true since the $\Omega^{(4,0)}$-component of $\ud H$ vanishes. 
The analog considerations are true for $(C\zo_X)_{H^{0,3}}$.

It is clear that $(T_X\oplus T^*_X)\otimes\C$ can be twisted by $H$.  
Straightforward computations show that this induces the twisted standard brackets 
on $C\oz_X$ and $C\zo_X$ respectively. Also the connections are induced by this Courant bracket.
\end{proof}

\subsection{Holomorphic Courant algebroids}\label{ss:holCour}
Let $X$ be a complex manifold.  We denote by $C\oz_X=T_X\oz\oplus(T_X\oz)^*$ 
and $C\zo_X=T_X\zo\oplus(T_X\zo)^*$ the  standard Courant algebroid.

\begin{defn} 
A \emph{holomorphic Courant algebroid} consists of a holomorphic vector bundle $E$ over a complex manifold $X$, 
with sheaf of holomorphic sections $\shs{E}$, endowed with a fiberwise $\CC$-valued inner product 
$\ip{~}{~}$  inducing a homomorphism of sheaves of $\hFF$-modules $\shs{E}\otimes_{\hFF}\shs{E}\xto{\ip{~}{~}}\hFF$, 
a bundle map $E\to T_X$ inducing a homomorphism of sheaves of $\hFF$-modules $\shs{E}\xto{\rho}\hXX$, 
where $\hXX$ denotes the sheaf of holomorphic sections of $T_X$, 
and a homomorphism of sheaves of $\CC$-modules $\shs{E}\otimes_{\CC}\shs{E}\xto{\db{}{}}\shs{E}$ satisfying relations 
\eqref{Jacobi}, \eqref{Leibniz}, \eqref{nSkew}, and \eqref{adInvar},
where $\D=\rho^*\circ\partial:\hFF\to\shs{E}$ and $\phi,\phi_1,\phi_2,\phi'\in\shs{E}$, $f\in\hFF$.
\end{defn} 

Holomorphic Courant algebroids were also studied by
Gualtieri, and we refer to~\cite{Gua10} for more details. 
\begin{lem}[Theorem~2.6.26 in~\cite{Huy05}]\label{wind}
Let $E$ be a complex vector bundle over a complex manifold $X$, and let $\shs{E}$ be 
a sheaf of $\hFF$-modules of sections of $E\to X$ such that, for each $x\in X$, 
there exists an open neighborhood $U\in X$ with $\sections{U;E}=\cinf{U,\CC}\cdot\shs{E}(U)$. 
Then the following assertions are equivalent. 
\begin{enumerate}
\item The vector bundle $E$ is holomorphic with sheaf of holomorphic sections $\shs{E}$. 
\item There exists a (unique) flat $T_X\zo$-connection $\nabla$ on $E\to X$ such that 
\[ \shs{E}(U)=\{\sigma\in\sections{U;E} \text{ s.t. } \conn_Y\sigma=0,\;\forall Y\in\XX\zo(X) \} .\]
\end{enumerate} 
\end{lem}

\begin{lem}\label{dream}
Let $E\to X$ be a holomorphic vector bundle with $\shs{E}$ as sheaf of holomorphic sections, 
and $\nabla$ the corresponding flat $T_X\zo$-connection on $E$.  Let $\ip{~}{~}$ be a 
smoothly varying $\CC$-valued fiberwise  symmetric nondegenerate
 $\CC$-bilinear form on  $E$. The following assertions are equivalent: 
\begin{enumerate}
\item The inner product $\ip{.}{.}$ induces a homomorphism of sheaves of $\hFF$-modules 
$\shs{E}\otimes_{\hFF}\shs{E}\to\hFF$. 
\item For all $\phi,\psi\in\sections{E}$ and $Y\in\XX\zo(X)$, we have 
\[ Y\ip{\phi}{\psi}=\ip{\conn_Y\phi}{\psi}+\ip{\phi}{\conn_Y\psi} .\] 
\end{enumerate}
\end{lem}

\begin{prop}\label{child}
Let $E\to X$ be a holomorphic vector bundle with $\shs{E}$ as the sheaf of holomorphic sections, 
and $\nabla$ the corresponding flat $T_X\zo$-connection on $E$. 
Let $\rho$ be a homomorphism of (complex) vector bundles from $E$ to $T_X$. 
The following assertions are equivalent:
\begin{enumerate}
\item The homomorphism $\rho$ induces a homomorphism of sheaves of $\hFF$-modules $\shs{E}\to\hXX$. 
\item The homomorphism $\rho\oz:E\to T_X\oz$ obtained from $\rho$ by identifying $T_X$ to $T_X\oz$ 
satisfies the relation \begin{equation}\label{zzzb} \rho\oz(\conn_Y \phi)=\proz\lie{Y}{\rho\oz\phi} , 
\quad \forall Y\in\sections{T_X\zo},\;\phi\in\sections{E\oz} .\end{equation}
\end{enumerate}
\end{prop}

\begin{lem}  Let $(E,\ip{.}{.},\rho,\db{.}{.})$ be a holomorphic Courant algebroid over a complex manifold $X$. 
Denote the sheaf of holomorphic sections of the underlying holomorphic vector bundle by $\shs{E}$ 
and the corresponding $T_X\zo$-connection on $E$ by $\nabla$. 
Then, there exists a unique complex Courant algebroid structure on $E\to X$ with inner product $\ip{~}{~}$ and anchor map 
$\rho\oz=\thalf(\id-i j)\circ\rho\:E\to T_X\oz\subset T_X\otimes\CC$, 
the restriction of whose Dorfman bracket to $\shs{E}$ coincides with $\db{}{}$.
Such a complex Courant algebroid is denoted by $E^{1,0}$.
\end{lem}

\begin{proof}  To prove the uniqueness, assume that there are two complex Courant algebroid structures 
on the smooth vector bundle $E\to X$ whose anchor map and inner products coincide. 
Moreover the Dorfman brackets coincide on $\shs{E}$, the sheaf of holomorphic sections. 
Then the Dorfman brackets coincide on all of $\Gamma(E)$, because the holomorphic sections 
over a coordinate neighborhood of $X$ are dense in the set of smooth sections.

To argue for the existence, note that we want the Leibniz rules
\begin{align*}
  \phi\dia(g\cdot\psi) &= \rho\oz(\phi)[g]\cdot\psi +g\cdot(\phi\dia\psi) \\
  (f\cdot\phi)\dia\psi &= -\rho\oz(\psi)[f]\cdot\phi +f(\phi\dia\psi) +\<\phi,\psi\>\cdot\rho^{(1,0)*}\ud f
\end{align*} for $\phi,\psi\in\Gamma(E)$ and $f,g\in\smooth(X)$.  Therefore 
we define the Dorfman bracket of the smooth sections $f\cdot\phi$ and $g\cdot\psi$,  $\forall \phi,\psi\in\shs{E}$,  as
\begin{equation*}
(f\cdot\phi)\dia(g\cdot\psi) = f\rho\oz(\phi)[g]\cdot\psi+fg\cdot(\phi\dia\psi) 
-g\rho\oz(\psi)[f]\cdot\phi +\<\phi,\psi\>g\cdot\rho^{(1,0)*}\ud f \;.
\end{equation*}
First we should argue that this extension is consistent with the bracket defined for holomorphic sections.  
Assume therefore that $f\cdot\phi$ is again holomorphic, where $f\in\smooth(X)$ and $\phi\in\shs{E}$.  
But this implies that $f$ is holomorphic on the open set where $\phi$ is not zero.  
Therefore on this open set the extension says
\begin{align*}
  f\rho\oz(\phi)[g]\cdot\psi &+fg\cdot(\phi\dia\psi) -g\rho\oz(\psi)[f]\cdot\phi +\<\phi,\psi\>g\cdot\rho^{(1,0)*}\ud f \\
  &= f\rho(\phi)[g]\cdot\psi +fg\cdot(\phi\dia\psi) -g\rho(\psi)[f]\cdot\phi +\<\phi,\psi\>g\cdot\rho^*\partial f
\end{align*}  
where we have used the fact that the two anchor maps coincide for holomorphic sections 
and $\rho^{(1,0)*}\ud f=\rho^*\partial f$.  Therefore the term is just $(f\cdot\phi)\dia(g\cdot\psi)$.

Since the first derivatives of $f$ and $g$ are continuous, 
the above relation also holds on the closure of the open domain where both $\phi$ and $\psi$ 
do not vanish. But on the complement, which is open, this term is just $0$, 
because at least one of the terms $\phi$ or $\psi$ as well as the bracket $\phi\dia\psi$ is also $0$.

 It follows from  a straightforward but tedious computation that 
 this bracket fulfills the four axioms \eqref{Jacobi}--\eqref{adInvar}.  
As an example, we show the proof of the axiom \eqref{nSkew}:
\begin{align*}
 (f\cdot\phi)\dia(f\cdot\phi) &= f^2\cdot(\phi\dia\phi) +\<\phi,\phi\>f\cdot\rho^{(1,0)*}\ud f \\
  &= \tfrac12\rho^*\partial\<\phi,\phi\> +\<\phi,\phi\>f\cdot\rho^{(1,0)*}\ud f  \\
  &=\tfrac12\rho^{(1,0)*}\ud\<f\cdot\phi,f\cdot\phi\>
\end{align*}  
where, in the last step,  we used again the fact that $\rho^*\partial=\rho^{(1,0)*}\ud$ when applied to holomorphic functions.
\end{proof}

Define a flat $C_X\zo$-connection $\nablaright$ on $E$ by 
\beq{squalor} \rconn_{Y\oplus\eta} e=\conn_Y e ,\eeq
and a flat $E$-connection $\nablaleft$ on $C_X\zo$ by 
\beq{filth} \lconn_{e} (Y\oplus\eta)=\nabla^o_{\rho\oz(e)}(Y\oplus\eta)=\przo\big(\mathcal{L}_{\rho\oz(e)}(Y\oplus\eta)\big) 
=\przo\big(\lie{\rho\oz(e)}{Y}\oplus\mathcal{L}_{\rho\oz(e)}\eta\big) ,\eeq
where $\przo$ denotes the projection of $(T_X\oplus T_X^*)\otimes\CC$ onto $T_X\zo\oplus(T_X\zo)^*$.

We can write the identity \eqref{zzzb} as
\begin{equation}\label{zzzbb} \rho(\rconn_Y\phi)-\lconn_\phi Y = [Y,\rho(\phi)] \;.
\end{equation}

Thus we have  the following

\begin{prop}  Let $(E,\ip{.}{.},\rho,\db{}{})$ be a holomorphic Courant algebroid over a complex manifold $X$.  
Denote the sheaf of holomorphic sections by $\shs{E}$ and the corresponding flat $T\zo_X$-connection on $E$ by $\nabla$.
Then the two complex Courant algebroids $C_X\zo$ and $E\oz$, together with the 
flat connections $\nablaleft$ and $\nablaright$ given by \eqref{squalor} and \eqref{filth}, 
constitute a matched pair of complex Courant algebroids, 
which we call the companion matched pair of the holomorphic Courant algebroid $E$.
\end{prop}

\begin{proof} 
We check that $C_X\zo\oplus E^{1,0}$ is a Courant algebroid. 
The axioms \eqref{Leibniz}--\eqref{adInvar} are straightforward computations 
by using the fact that the connections preserve the inner product on each Courant algebroid.

It remains to check that the Jacobi identity holds.
By Proposition~\ref{glass}, it suffices to check that the five properties 
\eqref{derofBr1}--\eqref{new4Y} hold.  All equations are trivially satisfied when $a_i\in\shs{E}$ 
and $\alpha_i\in\bar{\Theta}\oplus\bar\OO_X\subset \Gamma(C\zo_X)$. 
So we are left to check that multiplying the $a_i$ (or the $\alpha_i$ respectively) with smooth functions 
we get the same additional terms on each side of the equations. 
This will be demonstrated for \eqref{derofBr1}. The left hand side of \eqref{derofBr1} is
\begin{equation*}
\LHS(a_2) := \lconn_\alpha(a_1\dia a_2) - (\lconn_\alpha a_1)\dia a_2 
-a_1\dia(\lconn_\alpha a_2) -\lconn_{\rconn_{a_2}\alpha} a_1 +\lconn_{\rconn_{a_1}\alpha} a_2 
.\end{equation*}
Then 
\[ \LHS(f\cdot a_2)=f\cdot \LHS(a_2)+\left(\left[\rho(\alpha),\rho(a_1)\right]
+\rho(\rconn_{a_1}\alpha)-\rho(\lconn_\alpha a_1)\right)[f]\cdot a_2 \] 
and the second term vanishes due to \eqref{zzzbb}. 
For the right hand side 
\[ \RHS(a_2) := -\w(\alpha,\oo(a_1,a_2)+\thalf\ud\<a_1,a_2\>) -\thalf\D\<\alpha, \oo(a_1,a_2)
+\thalf \ud\<a_1,a_2\>\> ,\] 
we have 
\[ \RHS(f\cdot a_2) = f\cdot \RHS(a_2) .\]
Thus the right hand side is also $\smooth(X)$-linear in $a_2$.  Analog considerations result in coinciding terms 
for both sides of \eqref{derofBr1} when multiplying $a_1$ or $\alpha$ by a smooth function.
\end{proof}

In fact, the converse is also true.

\begin{prop} 
Let $X$ be  a complex manifold. Assume that
$(C\zo_X,B)$ is a   matched pair of complex Courant algebroids 
 such that the anchor of $B$ takes values in $T_X\oz$,
both connections $\nablaleft$ and $\nablaright$ are flat with
the $B$-connection $\nablaleft$ on $C\zo_X$ 
being given by $\lconn_{e} (Y\oplus\eta)=\nabla^o_{\rho(e)}(Y\oplus\eta)$.
Then there is a unique holomorphic Courant algebroid $E$ such that $B=E^{1, 0}$.
\end{prop}

\begin{proof} 
The flat $C\zo$-connection induces a flat $T\zo$-connection on $B$. 
Hence there is a holomorphic vector bundle $E\to X$ such that $E\oz=B$, 
according to Lemma~\ref{wind}.  
Since the connection $\nablaleft$ preserves the inner product on $B$ and is compatible 
with the anchor map \eqref{zzzb}, $E$ inherits a holomorphic inner product 
and a holomorphic anchor map.  It remains to check that the induced Dorfman bracket is holomorphic 
as well. If $a_1,a_2\in\shs{E}$ are two holomorphic sections of $E$ 
and $\alpha\in \bar{\Theta}_X\oplus\bar{\Omega}_X$ is an anti-holomorphic section of $C\oz_X$, 
then all terms of Equation~\eqref{derofBr1}, except the first one, vanish. 
But then the first one $\lconn_\alpha(a_1\dia a_2)$ also has to vanish. 
This shows that the Dorfman bracket of two holomorphic sections is itself holomorphic.
\end{proof}

\subsection{Flat regular Courant algebroid}\label{ss:flatReg}
A Courant algebroid $E$ is said to be \textbf{regular} if $F:=\rho(E)$ has constant rank, 
in which case $\rho(E)$ is an integrable distribution on the base manifold $M$ 
and $\ker\rho/(\ker\rho)\ortho$ is a bundle of quadratic Lie algebras over $M$. 
It was proved by Chen et al.~\cite{CSX09}
that the vector bundle underlying a regular Courant algebroid $E$ is isomorphic to
$\Fs\oplus \GG\oplus F$, where $F$ is the integrable subbundle $\rho(E)$ of $TM$ 
and $\GG$  the bundle $\ker\rho/(\ker\rho)^\perp$ of quadratic Lie algebras over $M$.
Thus we can confine ourselves to 
those Courant algebroid structures on $\fsgf$ whose anchor map is
\begin{equation} \label{Eqt:Standardrho} \rho(\xi_1+\gr_1+x_1)=x_1 ,\end{equation}
whose pseudo-metric is
\begin{equation} \label{Eqt:Standardip}
\ip{\xi_1+\gr_1+\fx_1}{\xi_2+\gr_2+\fx_2} = \duality{\xi_1}{\fx_2}
+\duality{\xi_2}{\fx_1}+\ipG{\gr_1}{\gr_2} 
,\end{equation}
and whose Dorfman bracket satisfies
\begin{equation} \label{Eqt:prGdb} \ProjectionTo{\GG} (\db{\gr_1}{\gr_2})=\lbG{\gr_1}{\gr_2} ,\end{equation}
where $\xi_1,\xi_2\in\Fs$, $\gr_1,\gr_2\in\GG$, and $\fx_1,\fx_2\in F$.
We call them \emph{standard} Courant algebroid structures on $\fsgf$.

\begin{thm}[\cite{CSX09}]
\label{Thm:StandardCourantStructure}
A Courant algebroid structure on $\fsgf$,
with pseudo-metric \eqref{Eqt:Standardip} and anchor map
\eqref{Eqt:Standardrho}, and satisfying \eqref{Eqt:prGdb}, is
completely determined by 
an $F$-connection $\nabla$ on $\GG$,
a bundle map $\Curvature:\wedge^2 F\to\GG$,
and a $3$-form $\HForm\in \sections{\wedge^3 \Fs}$
satisfying the compatibility conditions
\begin{gather}\label{ipGinvariant}  \ld{ x
}\ipG{\gr}{\gs}=\ipG{\conn_{x}{\gr}}{\gs}+\ipG{\gr}{\conn_{x}{\gs}},\\
\label{lbGinvariant}  \conn_{x}\lbG{\gr}{\gs}=
\lbG{\conn_x\gr}{\gs}+\lbG{\gr}{\conn_x\gs},\\
\label{dCurvautre0}
 \big(\conn_x\Curvature(y,z)-\Curvature(\lb{x}{y},z)\big) +c.p.=0,\\
\label{CurvatureConnectionDer} \conn_x\conn_y
\gr-\conn_y\conn_x \gr
-\conn_{\lb{x}{y}}\gr=\lbG{\Curvature(x,y)}{\gr}, \\
\label{dH}  \dF\HForm=\PontryaginCocycleR
\end{gather}
for all $x,y,z\in \sections{F}$ and $\gr,\gs\in\sections{\GG}$.
Here $\PontryaginCocycleR$ denotes the $4$-form on $F$ given by
\[ \PontryaginCocycleR(x_1,x_2,x_3,x_4) = \tfrac{1}{4} \sum_{\sigma\in S_4} \sgn(\sigma)
\ipG{\Curvature(x_{\sigma(1)},x_{\sigma(2)})}{\Curvature(x_{\sigma(3)},x_{\sigma(4)})}, \]
where $x_1,x_2,x_3,x_4\in F$.
\end{thm}

The Dorfman bracket on  $E=\fsgf$ is then given by
\begin{align}
\label{FSGFx1x2}
\db{\fx_1}{\fx_2}&=\HForm(\fx_1,\fx_2,\_)  + \Curvature
(\fx_1,\fx_2)+\lb{\fx_1}{\fx_2}, \\\label{FSGFg1g2}
\db{\gr_1}{\gr_2}&=\PForm
(\gr_1,\gr_2)+\lbG{\gr_1}{\gr_2},\\\label{FSGFxi1g2}
\db{\xi_1}{\gr_2}&=\db{\gr_1}{\xi_2}=\db{\xi_1}{\xi_2}=0,\\\label{FSGFx1xi2}
\db{\fx_1}{\xi_2}&=\LieDer_{\fx_1}\xi_2, \\\label{FSGFxi1x2}
\db{\xi_1}{\fx_2}&=-\LieDer_{\fx_2}\xi_1+
\dF\duality{\xi_1}{\fx_2},\\\label{FSGFx1g2}
\db{\fx_1}{\gr_2}&=-\db{\gr_2}{\fx_1}=-2\mathcal{Q}(\fx_1,\gr_2)+
\conn_{\fx_1}\gr_2,
\end{align}
for all $\xi_1,\xi_2\in\sections{\Fs}$, $\gr_1,\gr_2\in\sections{\GG}$,
$\fx_1,\fx_2\in\sections{F}$.
Here $\dF:\cinf{M}\to\sections{F^*}$ denotes the leafwise de~Rham differential.
The maps $\PForm:\sections{\GG}\otimes\sections{\GG}\to\sections{\Fs}$ and
$\mathcal{Q}: \sections{F}\otimes\sections{\GG}\to
\sections{F^*}$ are defined, respectively, by the relation
\begin{equation} \label{Eqt:definationofPForm}
\duality{\PForm(\gr_1,\gr_2)}{y} = 2\ipG{\gr_2}{\conn_y\gr_1} 
\end{equation}
and
\begin{equation} \label{Eqt:definationofQ}
\duality{\mathcal{Q}(x,\gr)}{y} = \ipG{\gr}{\Curvature(x,y)} 
.\end{equation}

\begin{defn}
A standard Courant algebroid  $E=\fsgf$ is said to be flat
if $\PontryaginCocycleR$ vanishes.
\end{defn}

\begin{prop}
Let $E=\fsgf$ be a flat
standard Courant algebroid. Then $(F\oplus F^*)_{H}$ and $\GG$
form a matched pair of Courant algebroids, 
and  $E$ is isomorphic to  their matched sum, where
$(F\oplus F^*)_{H}$ denotes the twisted Courant algebroid
on $F\oplus F^*$ by the $3$-form $H$.  Here
the $(F\oplus F^*)_{H}$ connection on $\GG$ is given by
\beq{ray}\rconn_{x+\xi}\gr=\conn_x\gr \;,\eeq
while the $\GG$-connection on $(F\oplus F^*)_{H}$ is given by
\beq{manta} \lconn_r (X\oplus\xi) = 2\mathcal{Q}(X,r)\oplus0 \;. \eeq
\end{prop}

\begin{proof} It follows from  the flatness of the Courant algebroid
that  the 3-form $H\in\OO^3_M(F)$ is $d_F$ closed. 
Therefore we can construct the twisted standard Courant algebroid $(F\oplus F^*)_H$.  
By comparing  each component of the Dorfman bracket of the sum Courant algebroid 
$\GG\oplus(F\oplus F^*)_H$ with that of the bracket \eqref{FSGFx1x2}--\eqref{FSGFx1g2},
we see that, for our choice \eqref{ray} and \eqref{manta}, both coincide.
\end{proof}

\section{Matched pairs of Dirac structures}\label{s:mpDirac}
Recall that a Dirac structure $D$ in a Courant algebroid $(E,\<.,.\>,\dia,\rho)$ 
with split signature is a maximal isotropic and integrable subbundle.

\begin{prop}
Given a matched pair $(E_1, E_2)$ of Courant algebroids of split signature and Dirac structures 
$D_1\subset E_1$ and $D_2\subset E_2$,
the direct sum $D_1\oplus D_2$ is a Dirac structure in the Courant algebroid
$E_1\oplus E_2$ iff $\lconn_\alpha a\in\sections{D_1}$ 
and $\rconn_a \a\in\sections{D_2}$ for all $\alpha\in\sections{D_2}$ and $a\in\sections{D_1}$.
\end{prop}

\begin{proof} 
It is obvious that $D_1\oplus D_2$ is maximal isotropic.
It remains to check that  the $E_2$ ($E_1$)-component of the bracket
of any two sections of $D_1$ is automatically in $D_2$ ($D_1$).
Indeed, we have
\[ \<0\oplus\alpha,(a\oplus0)\dia(b\oplus0)\> = \<\lconn_\alpha a,b\>_1 \;. \]
Since $D_1$ is isotropic, the RHS vanishes.
It thus follows from the maximal isotropy of $D_2$ that $(a\oplus0)\dia(b\oplus0)$ is in $D_2$.
\end{proof}

\begin{defn}
Let $(E_1,E_2)$ be a matched pair of Courant algebroids. 
A Dirac structure $D_1$ in $E_1$ and a Dirac structure $D_2$ in $E_2$ are said to form a matched pair of Dirac structures 
if their direct sum $D_1\oplus D_2$ is a Dirac structure in the matched sum $E_1\oplus E_2$. 
\end{defn}

\begin{cor} 
Let $D_1$ (resp.\ $D_2$) be a Dirac structure in a Courant algebroid $E_1$ (resp.\ $E_2$). 
If $(E_1,E_2)$ is a matched pair of Courant algebroids and $(D_1,D_2)$ is a matched pair of Dirac structures, 
then $(D_1,D_2)$ is a matched pair of Lie algebroids. 
\end{cor}

\begin{ex}  
Let  $CM:=TM\oplus T^*M$ be the standard Courant algebroid, $V\to M$
a vector bundle  with a flat connection $\nabla$.
Endowing  $V^*$ with the dual connection,
$CM$ and $V\oplus V^*$ are   matched pair of Courant algebroids.
\newline
Let  $\omega\in\Omega^2(M)$ and $L\in\sections{\wedge^2V^*}$.
$\Graph\omega\subset CM$ is a Dirac structure  in $CM$ iff
$\ud\omega=0$. On the other hand,
$\Graph L^\#\subset V\oplus V^*$ is automatically
a Dirac structure. Then $(\Graph\omega, \Graph L^\#)$ is a matched 
pair of Dirac structure iff
$[\conn_X,L^\#]=0$ for all $X\in\sections{TM}$.
In this case the direct sum Dirac structure 
is the  graph of  bundle map
\[ TM\oplus V \xto{\begin{pmatrix}\omega^\#&0\\ 0&L^\#\end{pmatrix}} T^*M\oplus V^* .\] 
On the other hand, we can consider the Dirac structure on
$CM$ given by the graph of  a Poisson bivector $\pi$ on $M$,
and the Dirac structure on $V\oplus V^*$ given by the graph of $\Lambda\in\sections{\wedge^2 V}$.
They form a matched  pair of Dirac structures if and only if $[\pi,\Lambda]_\oplus=0$
\end{ex}

%\appendix

\bibliography{grutzmann}
\bibliographystyle{amsplain}

\end{document}